\documentclass[leqno,a4paper]{article}
\usepackage{amsmath,amsthm,amssymb}

\newcommand{\R}{{\mathbb R}}
\newcommand{\K}{{\mathbb K}}

\newcommand{\m}{{\mu}}

\newcommand{\cq}{{\widehat{q}}}

\newcommand{\tc}{{\widetilde{c}}}

\newtheorem{remark}{Remark}[section]
\newtheorem{theorem}{Theorem}[section]
\newtheorem{lemma}[theorem]{Lemma}
\newtheorem{proposition}{Proposition}[section]

\numberwithin{equation}{section}
\def\vs1{\vspace{1ex}}

\def\O{\Omega}
\def\pa{\partial}
\def\dy{\displaystyle}

\def\be{\begin{equation}}
\def\ba{\begin{array}}
\def\ea{\end{array}}
\def\ee{\end{equation}}

\begin{document}
\title{\bf\large Singular parabolic p-Laplacian systems under non-smooth external forces. Regularity up to the boundary.}
\author{ H.~Beir\~ao da Veiga}

\date{}
\maketitle
\begin{abstract} We study the regularity of the solutions to initial-boundary
value problems for $N-$systems of the p-Laplacian type, in
$\,n\geq\,3\,$ space variables, with square-integrable external
forces in the space-time cylinder. So, the ellipticity coefficient
remains unbounded. The singular case $\mu=0\,$ is covered.
\end{abstract}

\vspace{0.2cm}

\noindent \textbf{Keywords:} Initial-boundary value problems,
p-Laplacian parabolic singular systems, regularity up to the
boundary.


\section{Introduction and main result.}
In the sequel we consider the evolution problem
\begin{equation}\label{NSC-velh}\left\{
\begin{array}{ll}\vspace{1ex}
\pa_t\,u -\,\nabla \cdot
\,\big(\,(\,\m+|\,\nabla\,u|\,)^{p-2}\,\nabla\,u\,\big)=\,f(t,\,x)\,, \
\mbox{ in } (0,\,T) \times\,\O\,,
\\%
u=\,0\ \mbox{ on } (0,\,T) \times\, \partial \O\,,
\\%
u(0)=\,u_0 \ \mbox{ in } \O\,,
\end{array}\right .
\end{equation}
where $\,p\in\,(1,\,2]\,,$ $T\in(0,\,\infty]\,,$ and $\mu\,\geq \,0$
are constants. Here $\,u\,$ is an $\,N-$dimensional vector field,
$\,N\geq\,1\,,$ defined in  $\,Q_T \equiv \,(0,\,T) \times\,\O\,$
where $\,\O \subset \R^n$, $\,n\geq 3\,,$ is a regular, bounded open
set. The main point here is that the external force $f$ is only
square-integrable in $\,Q_T\,$. This low integrability prevents
boundedness of $\,|\,\nabla\,u(t,x)|\,$ (which holds, for instance,
if $\,f \in \,L^q(Q_T)\,$, $\,q> \,n+\,2\,).$ So, the ellipticity
coefficient $\,(\,(\,\m+|\,\nabla\,u|\,)^{p-2}\,$ keeps unbounded.
This obstacle is here by-passed, due to a simple, but fruitful, idea
(which has a more wide range of application, as shown in a
forthcoming work).\par%
Let us illustrate the kind of results proved in the sequel, by the
following example. Assume that $\,p\,$ satisfies the condition
\eqref{oras}, where $\,K\,$ (see below) is a positive constant,
independent of $\,p\,$. Then, the second order space derivatives
satisfy
\begin{equation}\label{maxreg-2}%
D^2 u \in \,L^{2\,(p-\,1\,)}(0,\,T;\,L^{\cq}(\O)\,)\,,%
\end{equation}
where $\,\cq\,$ is defined by \eqref{errq2}.%

\vspace{0.2cm}

The proof of our main result appeals to a regularity theorem, see
the theorem \ref{teoremaq} below, proved in reference \cite{BVCRI}
for the stationary problem
\begin{equation}\label{NSC}\left\{
\begin{array}{ll}\vspace{1ex}
-\,\nabla \cdot
\,\big(\,(\,\m+|\,\nabla\,u|\,)^{p-2}\,\nabla\,u\,\big)=\,f \ \mbox{
in } \O\,,
\\%
u=\,0\ \mbox{ on } \partial \O\,.
\end{array}\right .
\end{equation}
Actually, we need the above regularity result for the value
$\,q=\,\cq<\,2\,,$ see \eqref{errq2} below. However, values smaller
then $\,2\,$ are out of the range considered in the statement of
theorem \ref{teoremaq}, even though the proof applies to a range of
values which includes $\,\cq<\,2\,.$ The check of this claim is
straightforward. However, for the readers convenience, after the
statement of the extension result to the $\,\cq<\,2\,$ case (see
proposition \ref{propmaq} below) we made a couple of comments
plentifully sufficient to adapt the proof given in \cite{BVCRI} to
the $\,\cq$ case. To minimize the number of alterations, assume that
\begin{equation}\label{bunov}\left\{
\begin{array}{ll}\vspace{1ex}
\frac{2\,n}{n+\,2}<\,p \leq\,2\,, \quad \ \mbox{ if }\, n>\,3\,,
\\%
\frac54 <\,p\,, \quad \mbox{ if }\,  n=\,3\,.
\end{array}\right .
\end{equation}
In particular, the inclusion
\begin{equation}\label{doispp}
\,L^2(\O)\,\subset \,\,W^{-1,\,p'}(\O)\
\end{equation}%
holds. Actually, the assumption \eqref{bunov} is not strictly
essential in the following.

\vspace{0.2cm}

We start by recalling that scalar multiplication of both sides of
\eqref{NSC-velh} by $\,u\,$, followed by classical manipulations,
lead to the well known a priori estimate
$$
\|\,u\,\|^2_{L^{\infty}(0,\,T;\,L^2(\O)\,)}\,
+\,\|\,u\,\|^p_{L^p(0,\,T;\,W^{1,\,p}(\O)\,)}
$$
$$
\leq\,c\,\big(\,\|\,u_0\,\|^2_{L^2(0,\,T;\,L^2(\O)\,)}+\,\|\,f\,\|^{p'}_{
L^{p'}(0,\,T;\,W^{-1,\,p'}(\O)\,)}\,\big).
$$
This estimate, useful in proving the existence of the weak
solution, lead us to assume that%
\begin{equation}\label{ppprimo}%
f\in L^{p'}(\,0,\,T;\,W^{-1,\,p'}(\O)\,)\,.
\end{equation}
For the existence of the above weak solution we refer the reader to
the Theorem 1.1, Chap. II, in \cite{lions}.

\vspace{0.2cm}

Let us introduce the core exponent
\begin{equation}\label{errq2}%
\cq=\,\frac{2\,n\,(\,p-\,1\,)}{n-\,2\,(\,2-\-p\,)}\,.
\end{equation}
As shown below, the central role of this exponent is due to the
particular relation
\begin{equation}\label{errq}%
r(\cq)=\,2\,,%
\end{equation}
see \eqref{rq}. Our assumptions on $\,p\,$ implies that
$\,\cq\in\,(1,\,2]\,,$ and also that the immersion
\begin{equation}\label{doispp}
\,W^{2,\,\cq}(\O)\,\subset \,\,W^{1,\,p}(\O)
\end{equation}%
is compact.\par%
Finally we recall the well known
inequality%
\be\label{ladaq}\|D^2\,v\|_{q}\leq \,C_2(q)\,\|\Delta v\|_q\,,\ee%
for $v\in W^{2,q}(\O)\cap W_0^{1,q}(\O)\,$. Actually, there is a
constant $K\,,$ independent of $q$, such that%
\be\label{yud} C_2(q) \leq\, K\, q\,,\ee%
at least for $\,q>\,\frac{2\,n}{n+\,2}\,$ (see \cite{yud}).\par%
Our main result is the following.
\begin{theorem}\label{isadeo}%
Let $\,p \,$ satisfy \eqref{bunov}, and define $\,\cq\,$ by
\eqref{errq2}. Further, assume that
\begin{equation}\label{kkapas}
\,(2-p)\,C_2(\cq)<\,1\,,
\end{equation}
where $C_2(\cq)\,$ is defined by \eqref{ladaq}. Let $\,u_0
\in\,W^{1,\,p}_0(\O)\,$ and assume that, for some $\,T \in
\,]\,0,\,+\,\infty\,]\,,$ $\,f\,$ satisfies \eqref{ppprimo} and
\begin{equation}\label{effes}%
f\in\,L^2(0,\,T;\,L^2(\O)\,)\,.
\end{equation}
Then the weak solution $\,u\,$ of problem \eqref{NSC-velh}
enjoys the following properties:%
\begin{equation}\label{zeras}
u \in \,L^\infty (0,\,T;\,W^{1,\,p}_0(\O)\,)\,,
\end{equation}
\begin{equation}\label{segas}%
\,\nabla \cdot \,\big(\,(\,\m+|\,\nabla\,u|\,)^{p-2}\,\nabla\,u\,
\big) \in\,L^2(0,\,T;\,L^2(\O)\,)\,,%
\end{equation}
\begin{equation}\label{tercas}%
\pa_t\,u \in\,L^2(0,\,T;\,L^2(\O)\,)\,,
\end{equation}
and
\begin{equation}\label{maxreg}%
u\in \,L^{2\,(p-\,1\,)}(0,\,T;\,W^{2,\,\cq}(\O)\,)\,.%
\end{equation}
\end{theorem}%

\vspace{0.2cm}

Note that, if $\,n=\,3\,$ and $\,p>\,\frac32\,$, then
$$
u\in
\,L^{2\,(p-\,1\,)}(0,\,\,T\,;\,C^{\,0,\,\alpha}(\overline{\Omega})\,)\,,
$$
where $\,\alpha=\,\frac{p-\,\frac32}{p-\,1}$.
\begin{remark}\label{emarq}
\rm{Due to \eqref{yud}, the condition \eqref{kkapas} holds if
\begin{equation}\label{bolas}%
(2-\,p) \cq <\,\frac{1}{K}\,,%
\end{equation}
at least by assuming $\,p > 2-\,\frac2n \,$ (as required by Yudovic
assumption on $n$, actually, not strictly necessary).
Further, straightforward calculations show that \eqref{bolas} holds if%
\begin{equation}\label{oras}%
2-\,\frac{n}{2nK+\,2}<\,p\leq \,2\,.
\end{equation}
So, \eqref{oras} by itself, is a sufficient condition to guarantee
the results claimed in theorem \ref{isadeo}. The main point is that
this condition depends only on $p\,$, via Yudovic's constant $K$.}
\end{remark}

\vspace{0.2cm}

Sharp estimates for the norms of the left hand sides of the above
equations, in terms of data norms and $\,\mu\,,$ follow immediately
from the proofs. For a more detailed discussion see sections
\ref{tres} and \ref{quatro}. We state here these estimates in the
singular case $\,\mu=\,0\,$. One has the following result:
\begin{theorem}\label{isadeo2}%
Consider the singular parabolic problem
\begin{equation}\label{NSC-velh-s}\left\{
\begin{array}{ll}\vspace{1ex}
\pa_t\,u -\,\nabla \cdot
\,(\,|\,\nabla\,u|^{p-2}\,\nabla\,u\,)=\,f(t,\,x)\,, \ \mbox{ in }
(0,\,T) \times\,\O\,,
\\%
u=\,0\ \mbox{ on } (0,\,T) \times\, \partial \O\,,
\\%
u(0)=\,u_0 \ \mbox{ in } \O\,.
\end{array}\right .
\end{equation}
Let the hypothesis assumed in Theorem \ref{isadeo}, concerning
$\,p\,$, $\,\cq\,,$ and $\,f\,,$ hold. Then, one has
\begin{equation}\label{primas2}
\frac2p \,\|\,\nabla \,u\,\|^p_{L^\infty (0,\,T;\,L^p(\O)\,)}\leq\,
\frac2p \, \|\,\nabla \,u_0\,\|^p_p + \,\|\,f\,\|^2_{L^2
(0,\,T;\,L^2(\O)\,)}\,,
\end{equation}%
\begin{equation}\label{segas2}
\|\,\nabla \cdot
\,(\,|\,\nabla\,u\,|^{p-2}\,\nabla\,u\,)\,\|^2_{L^2(0,\,T;\,L^2(\O)\,)}\leq\,
\frac2p \, \|\,\nabla \,u_0\,\|^p_p + \,\|\,f\,\|^2_{L^2
(0,\,T;\,L^2(\O)\,)}\,,
\end{equation}%
\begin{equation}\label{tercas2}%
\|\,\pa_t\,u\,\|^2_{L^2(0,\,T;\,L^2(\O)\,)}\leq\, \frac2p \,
\|\,\nabla \,u_0\,\|^p_p + \,2\,\|\,f\,\|^2_{L^2
(0,\,T;\,L^2(\O)\,)}\,,
\end{equation}
and
\begin{equation}\label{funds3}%
\ba{ll}\vs1\dy
\|\,u\,\|^2_{L^{2(p-\,1)}(\,0,\,T;\,W^{2,\,\cq}(\O)\,)\,} \leq\,C\,
T^{\frac{2-\,p}{p-\,1}}\Big(\,\|\,\nabla \,u_0\,\|^p_p +
\,\|\,f\,\|^2_{L^2 (0,\,T;\,L^2(\O)\,)}\,\Big)\,\\
\\
+\,C\,\Big(\,\|\,\nabla \,u_0\,\|^{\frac{p}{p-\,1}}_p +
\,\|\,f\,\|^{\frac{2}{p-\,1}}_{L^2 (0,\,T;\,L^2(\O)\,)}\,\Big)\,.%
\ea
\end{equation}
\end{theorem}%
It is worth noting that, for $\,p=\,2\,$, the above estimates turn
into the classical "heat equation" estimates. For instance,
\eqref{funds3} reduces to
\begin{equation}\label{fundp2}%
\|\,u\,\|^2_{L^2(\,0,\,T;\,W^{2,\,2}(\O)\,)\,}
\leq\,C\,\big(\,\|\,\nabla \,u_0\,\|^2_2 +\,\|\,f\,\|^2_{L^2
(0,\,T;\,L^2(\O)\,)}\,\big)\,.%
\end{equation}

\vspace{0.2cm}

For related results we refer, for instance, to the well-know
monographs \cite{dibenedetto},  \cite{Ladsolural}, \cite{lions}, and
to references \cite{Chen}, \cite{Choe}, \cite{dibenedetto},
\cite{diben-fried 1}, \cite{diben-fried 2}, \cite{diben-kvong-ve},
\cite{diben-kvong} \cite{lieberman 2005}, \cite{lieberman 86},
\cite{lieberman 93}.\par%
In references  \cite{diben-fried 1}, \cite{diben-fried 2} (see
\cite{dibenedetto} chapters IX, X) local H\" older continuity in $\,
(0,\,T) \times \,\O\,$ of the space gradient of local weak solutions
is proved. Regularity results, \emph{up to the boundary}, are stated
in the chapter X of \cite{dibenedetto} (see the Theorems 1.1 and 1.2
therein). In particular, in the Theorem 1.2, the H\" older
continuity up to the parabolic boundary (where $\,u=\,0\,$) of the
spatial gradient of weak solutions $\,u\,$ is proved. However,
regularity results, up to the boundary, for the second order space
derivatives in $\,L^q(\O)\,$ spaces, for solutions to the parabolic
singular system \eqref{NSC-velh-s}, were not know in the literature.
Actually, the two types of estimates are not comparable. It is worth
noting that in the elliptic case, see \cite{BVCRI}, the
$\,W^{2,\,q}(\O)\,$ estimates imply
$\,C^{1,\,\alpha}(\overline{\O})\,$ regularity, since $\,q>\,n\,$ is admissible.\par%
A classical related subject, in the case $\,N=\,1\,,$ are the
Harnack's inequalities. See references and results in the recent
monograph \cite{diben-ugo-vesp}.  We learned in reference
\cite{dibenedetto} that the first parabolic versions of Harnack's
inequality are due to Hadamard \cite{hadamard} and Pini \cite{pini}.%

\vspace{0.2cm}

NOTATION: We follow the notation introduced in reference
\cite{BVCRI}. By $L^p(\O)$ and $W^{m,p}(\O)$, $m$ nonnegative
integer and $p\in(1,+\infty)$, we denote the usual Lebesgue and
Sobolev spaces, with the standard norms $\|\cdot\|_{p}$ and
$\|\,\cdot\,\|_{m,p}\,$. We set $\|\cdot\|=\|\cdot\|_{2}$. We denote
by $W^{1,p}_0(\O)$ the closure of $C^\infty_0(\O)$ in
$W^{1,p}(\O)\,,$ and by $W^{-1,p'}(\O)$, $p'=\,p/(p-1)$, the strong
dual of $W^{1,p}_0(\O)$ with norm $\|\,\cdot\,\|_{-1,p'}$.
\par%
The symbols $c$, $c_1$, $c_2$, etc., denote positive constants that
may depend on $\mu$; by capital letters, $C$, $C_1$, $C_2$, we
denote positive constants independent of $\mu \geq\,0\,$(eventually,
$\,\mu\,$ bounded from above). The same symbol $c$ or $C$ may denote
different constants, even in the same equation. we set $\partial_t
\,u=\,\frac{\pa\, u}{\pa\, t}\,.$

\section{The stationary problem. Known results.}
As already referred, a main ingredient used here to prove the core
estimate \eqref{maxreg} concerns the stationary problem \eqref{NSC}.
The following result was proved in reference \cite{BVCRI}, in
collaboration with Francesca Crispo (we also recall the previous
work \cite{BDVCRIplap}, by the same authors).
\begin{theorem}\label{teoremaq}
Let $p\in (1,2]$ and $q \geq\,2\,,$ $\,q\not= n\,,$ be given. Assume
that $\,(2-p)\,C_2(q)<\,1\,$, where $C_2(q)$ satisfies
\eqref{ladaq}. Further, assume that $\mu\geq 0$. Let $f\in
L^{r(q)}(\O)\,,$ where%
\be\label{rq}r(q)=\left\{\begin{array}{ll}\dy
\frac{nq}{n(p-1)+q(2-p)}
&\dy \mbox{ if }\ q\in [\,2,\,n]\,,\\
\hskip1cm q & \dy \mbox{ if }\ q \geq\, n\,,
\end{array}\right .\ee
and let $u$ be the unique weak solution of problem \eqref{NSC}. Then
$u$ belongs to $W^{2,q}(\O)$. Moreover, the following estimate holds
\begin{equation}\label{dnq}
 \|u\|_{2,q}\leq C
 \,\left(\|f\|_q+\|f\|_{r(q)}^\frac{1}{p-1}\right)\,.
\end{equation}
\end{theorem}
The reader directly interested in the above result is refereed to
\cite{BVCRI}, where significance and range of
application of the above statement are discussed.\par%
In reference \cite{BVCRI} the authors assume that $\,q \geq\,2\,$
since they were mainly interested in maximal regularity. However
results and proofs hold also for values $\,q\,\in (1,\,2)\,,$ at
most under some small modification. Here we need the above result
only for the particular value $\,q=\,\cq<\,2\,$. For simplicity, we
take into account only this value, and show the single points in the
proof given in \cite{BVCRI} where some small remark
may be useful to adapt the proof to the value $\,\cq\,.$%
\begin{proposition}\label{propmaq}
Let be $\mu\geq 0\,,$ and let $\,p\,$ satisfy \eqref{kkapas} and
\eqref{bunov}. Assume that $f\in L^2(\O)\,.$ Then, the weak solution
$\,u\,$ to the problem \eqref{NSC} belongs to $W^{2,\cq}(\O)$.
Moreover,
\begin{equation}\label{dnq}
 \|u\|_{2,\cq}\leq C
 \,\left(\|f\|_\cq+\|f\|_2^\frac{1}{p-1}\right)\,.
\end{equation}
\end{proposition}
\begin{proof}
Clearly, we assume that the reader have in hands the proof of
theorem \ref{teoremaq} given in reference \cite{BVCRI}. The unique
real modification to be made in this proof, in order to adapt it to
the present situation, is the following. In reference \cite{BVCRI},
at the end of section 3, the authors prove the following convergence
\begin{equation}\label{vais}
f\left(\mu+|\nabla\,v^m|\right)^{2-p}\,\rightarrow
\,f\,\left(\mu+|\nabla\,v|\right)^{2-p}
\end{equation}
in the $\,L^{\frac{q}{2}}$ norm. This is not suitable here, since
$\,q=\,\cq <\,2\,.$ However, as remarked in \cite{BVCRI},
convergence in the distributional sense is obviously sufficient. As
in \cite{BVCRI}, one has
$$
|\,(\mu+|\nabla\,v^m|\,)^{2-p} -\,(\mu+|\nabla\,v|\,)^{2-p}\,|
\leq\,\frac{2-\,p}{\mu^{p-\,1}}\,|\nabla\,v^m -\,\nabla\,v|\,.
$$
In particular, it follows that \eqref{vais} holds a.e. in
$\,\O\,$.\par%
On the other hand, $\,\left(\mu+|\nabla\,v^m|\right)^{2-p}$ is
bounded in $\,L^t\,$, where
$\,t:=\,\frac{\cq*}{2-\,p}>\,\frac{n}{2}\,.\,$ So, it follows from
Lemma 1.3, in Chap.1, \cite{lions} that
$$
\left(\mu+|\nabla\,v^m|\right)^{2-p}\,\rightarrow
\,f\,\left(\mu+|\nabla\,v|\right)^{2-p}\,,%
$$
weakly in $\,L^t\,.$ Moreover, \eqref{bunov} implies that
$(\frac{n}{2})'\leq\,2\,.$ Consequently, $\,f \in
\,L^{(\frac{n}{2})'}\,.$
It readily follows that \eqref{vais} holds in the distributional sense.\par%
Just for the reader's convenience we add two (we believe
dispensable) remarks:\par%
i) The set
$$
\K=\{ v\in\,W^{2,\,\cq}(\O)\,:\,\| \Delta\,v\|_\cq \leq\,R\,,\, v=0\
\mbox{ on } \pa\O\}\,,
$$
introduced in \cite{BVCRI}, section 3, is still contained in
$\,W^{1,\,p}_0(\O)\,$, since $\,p\,>\frac{2\,n}{n+\,2}\,.$\par%
ii) As at the very beginning of section 4 in \cite{BVCRI}, one still
has here $\,p<\,q^*\,$. So, as in
\cite{BVCRI}, $\,u^\mu\,$ converges to $\,u\,$ in $W^{1,p}(\O)$.%
\end{proof}
\section{A more general setting.}\label{tres}
Besides proving the theorem \ref{isadeo}, we also want to show that
the statement may be extended to other systems of equations of the
form
\begin{equation}\label{NSC-nov}\left\{
\begin{array}{ll}\vspace{1ex}
\pa_t\,u -\,\nabla \cdot \,S(\,\nabla\,u\,) \,=\,f(t,\,x) \ \mbox{
in } (0,\,T\,) \times\,\O\,,
\\%
u=\,0\ \mbox{ on } (0,\,T) \times\, \partial \O\,,
\\%
u(0)=\,u_0 \ \mbox{ in } \O\,,
\end{array}\right .
\end{equation}
where $\,S(\cdot)\,$ is given by
\begin{equation}\label{SBSB}
 S(\,\nabla\,u\,):=\,B(\,|\nabla\,u|\,)\,\nabla\,u\,.
\end{equation}

Note that the $\,N-$dimensional vector field vector $ \nabla \cdot
\,S(\,\nabla\,u\,)$ has components given by

\begin{equation}\label{jote}
(\,\nabla \cdot \,S(\,\nabla\,u\,)\,)_j =\,\sum_i \,\pa_i
\big(\,B(\,|\nabla\,u|\,)\,\pa_i \,u_j \,\big)\,.
\end{equation}

In the sequel we show that the above extension is possible, provided
that proposition \ref{propmaq} applies to the corresponding
stationary problem
\begin{equation}\label{NSCnov}\left\{
\begin{array}{ll}\vspace{1ex}
-\,\nabla \cdot \,\big(\,B (\,|\nabla\,u|\,)\,\nabla\,u\,\big)=\,f \
\mbox{ in } \O\,,
\\%
u=\,0\ \mbox{ on } \partial \O\,.
\end{array}\right .
\end{equation}
This last possibility was, rightly, claimed in reference
\cite{BVCRI}. So, in this section, we consider the system
\eqref{NSC-nov}, and \emph{assume} that a suitable extension of
proposition \ref{propmaq} to the system \eqref{NSCnov} holds.
Further, we introduce the following notation, suited to treat the
general situation.\par%
We set, for $\,y >\,0\,,$
\begin{equation}\label{AAAs}
A(y):=\,B(\sqrt{y}\,)\,,
\end{equation}
and define, for $\,y\geq\,0\,,$
\begin{equation}\label{estrla}
G(y):= \int \,A(y)\,dy\,.
\end{equation}
Furthermore, we assume that there are positive constants $\,c_0\,$
and  $\,c_1\,$ such that
\begin{equation}\label{assumg}
c_0\,y^p -\,c_1 \leq \,G(y^2) \leq\,\tc_0\,y^p +\,\tc_1\,,
\end{equation}
for $\,y\geq\,0\,.$\par%
In this section we show the following result.
\begin{proposition}\label{teoras}%
Assume that the solutions to the stationary problem \eqref{NSCnov}
enjoy the regularity result stated in proposition \ref{propmaq}, and
that \eqref{assumg} holds. Further, assume that
$\,f\in\,L^2(0,\,T;\,L^2(\O)\,)\,$, for some
$\,T \in \,]0,\,+\,\infty\,]\,.$ Then%
\begin{equation}\label{primas}
c_0\,\|\,u\,\|^p_{L^\infty (0,\,T;\,W^{1,\,p}_0(\O)\,)}\leq\,\tc_0
\,\|\,u_0\,\|^p_{\,W^{1,\,p}_0(\O)}+\,\|\,f\,\|^2_{L^2
(0,\,T;\,L^2(\O)\,)}+\,(c_1+\,\tc_1)\,|\,\O\,|\,,
\end{equation}%
\begin{equation}\label{segasg}
\|\,\nabla\,\cdot\,S(\,\nabla\,u\,)\,\|^2_{L^2(0,\,T;\,L^2(\O)\,)}\leq\,\tc_0
\,\|\,u_0\,\|^p_{\,W^{1,\,p}_0(\O)}+\,\|\,f\,\|^2_{L^2
(0,\,T;\,L^2(\O)\,)}+\,(c_1+\,\tc_1)\,|\,\O\,|\,,
\end{equation}%
\begin{equation}\label{tercasg}%
\|\,\pa_t\,u\,\|^2_{L^2(0,\,T;\,L^2(\O)\,)}\leq\,\tc_0
\,\|\,u_0\,\|^p_{\,W^{1,\,p}_0(\O)}+\,2\,\|\,f\,\|^2_{L^2
(0,\,T;\,L^2(\O)\,)}+\,(c_1+\,\tc_1)\,|\,\O\,|\,,
\end{equation}
and
\begin{equation}\label{funds2}
\|\,u\,\|_{L^{2(p-\,1)}_T(W^{2,\,\cq}(\O)\,)} \leq\,C\Big(\,
T^{\frac{1}{2(p-\,1)}}+\, \tc_0^{\,\frac{1}{2(p-\,1)}}
\,\|\,u\,\|^{\frac{p}{2(p-\,1)}}_{\,W^{1,\,p}_0(\O)}+\,
\end{equation}
$$
\|\,f\,\|^{\frac{1}{(p-\,1)}}_{L^2_T(L^2(\O))}+\,\big(\,(c_1+\,\tc_1)\,|\,\O\,|\,\big)^{\frac{1}{2(p-\,1)}}\,\Big)\,.
$$
\end{proposition}

\begin{proof}
By scalar multiplication of booth sides of the first equation
\eqref{NSC-nov} by $-\,\nabla \cdot
\,\big(\,B(|\,\nabla\,u|)\,\nabla\,u\,\big)$, followed by
integration in $\,\O\,$, one gets
\begin{equation}\label{interm}
\ba{ll}\vs1\dy \frac12\,\int_\O \,B(|\,\nabla\,u|) \,\pa_t
\,|\,\nabla\,u|^2\, dx
+ \,\int_\O \,|\,\nabla \cdot \,S(\,\nabla\,u\,)\,|^2 \,
dx \,=\\
\\
-\,\int_\O \, f \cdot\,(\,\nabla \cdot \,S(\,\nabla\,u\,) \,)
\,dx\,.\ea
\end{equation}
Recall \eqref{jote}. We have appealed to an integration by parts and
to the fact that $\,\pa_t\,u =\,0\,$ on $\, \partial \O\,.$ Next, we
write the equation \eqref{interm} in the form
\begin{equation}\label{interm2}
\ba{ll}\vs1\dy \frac12\,\frac{d}{d\,t}\, \int_\O
\,G(|\,\nabla\,u|^2)\, dx + \,\int_\O \,|\,\nabla \cdot
\,S(\,\nabla\,u\,)\,|^2 \,
dx \,=\\
\\
-\,\int_\O \, f \cdot\,(\,\nabla \cdot \,S(\,\nabla\,u\,)\,)
\,dx\,,\ea
\end{equation}
Note that, if $\,f=\,0\,,$ the quantity
$$
\int_\O \,G(|\,\nabla\,u(t)|^2)\, dx
$$
is decreasing with respect to time. For instance, in the singular
case \eqref{NSC-velh-s}, the norm $\,\|\,\nabla \,u(t)\,\|_p\,$
decreases with time.\par%
From \eqref{interm2}, it follows that
\begin{equation}\label{interm3}
\frac{d}{d\,t}\, \int_\O \,G(|\,\nabla\,u|^2)\, dx + \,\int_\O
\,\big|\,\nabla \cdot \,\big(\,A(|\,\nabla\,u|^2)
\,\nabla\,u\,\big)\,\big|^2 \, dx \leq\, \int_\O \, |\,f\,|^2
\,dx\,.
\end{equation}
By integration with respect to $\,t\,$, one gets
\begin{equation}\label{itempo}
 \int_\O \, G(\,|\,\nabla\,u(t)|^2) dx +\,\int_0^t \,\big\|\,\nabla \cdot
\,S(\,\nabla\,u(s)\,)\,\big\|^2_2 \,ds\leq\, \int_\O \,
G(\,|\,\nabla\,u_0|^2) dx +\, \int_0^t \, \|\,f(s)\,\|^2_2 \,ds\,.
\end{equation}
From \eqref{itempo} and \eqref{assumg}, by appealing to well know
manipulations, one proves \eqref{primas}, \eqref{segasg}, and
\eqref{tercasg}. Note that \eqref{tercasg} follows immediately from
the identity
$$
\pa_t\,u =\,\nabla \cdot\,S(\,\nabla\,u\,)+\,f(t,\,x) \,.
$$
In the above estimates $\,W^{1,\,p}_0(\O)\,$ is endowed with the
norm $\,\|\,\nabla\,u\,\|_{\,L^p(\O)}\,.$\par%
Finally, we prove \eqref{funds2}. By assumption, the solutions to
the stationary problem \eqref{NSCnov} enjoy the regularity results
stated in proposition \ref{propmaq}. So, it follows from equations
\eqref{NSC-nov}, \eqref{dnq}, \eqref{segasg}, and \eqref{tercasg}
that
\begin{equation}\label{uns}%
u \in\,L^{2(\,p-\,1\,)}(0,\,T\,; W^{2,\,\cq}(\O)\,)\,.
\end{equation}%
More precisely, by appealing to \eqref{dnq}, one gets, for a.a.
$t\in (0,\,T)\,,$
$$
\|\,u(t)\|^{2(p-\,1)}_{2,\,\cq} \leq\,C\,\big(\,\|\,\pa_t\,u
-\,f\,\|^{2(p-\,1)}_2+\,\|\,\pa_t\,u -\,f\,\|^2_2\,\big)\,.
$$
Note that, for $\,p=\,2\,$, we get the classical estimate.\par%
Since $\,2(\,p-\,1\,) \leq\,2\,,$ we may replace the above estimate
simply by
$$
\|\,u(t)\|^{2(p-\,1)}_{2,\,\cq} \leq\,C\,\big(\,1+\,\|\,\pa_t\,u
-\,f\,\|^2_2\,\big)\,.
$$
So, with obvious notation,
\begin{equation}\label{funds}
\|\,u\,\|_{L^{2(p-\,1)}_T(W^{2,\,\cq}\,)} \leq\,
C\,T^{\frac{1}{2(p-\,1)}}+\,C\,\|\,\pa_t\,u
-\,f\,\|^{\frac{1}{p-\,1}}_{L^2_T(L^2\,)}\,.
\end{equation}
Hence, by the assumptions on $\,f\,$ together with \eqref{tercas},
one gets \eqref{funds2}. This completes the proof of proposition
\ref{teoras}.\par%
\end{proof}
Note that proposition \ref{teoras} is only a \emph{partial
extension} of theorem \ref{isadeo} to more general systems of the
form \eqref{NSC-nov} since, in this last case the proposition
\ref{propmaq} was not proved. However, the corresponding extension
should be routine.\par%
\section{Proof of Theorem  \ref{isadeo}.}\label{quatro}
To prove the theorem \ref{isadeo}, we simply assume that $\,B\,$ is
given by
\begin{equation}\label{svelho}
B(|\,\nabla\,u|)=\,(\,\m+|\,\nabla\,u|\,)^{p-2}\,.
\end{equation}
It remains to show that, in the case of equation \eqref{NSC-velh},
an estimate like \eqref{assumg} holds.
\begin{lemma}
Set, for $\,y \geq\,0\,$,
\begin{equation}\label{asinhas}
A(\,y)=\,(\,\m+\,y^\frac12 \,)^{p-2}\,.%
\end{equation}%
It follows that
\begin{equation}\label{ges}%
G(\,y^2)= \,\frac{2}{p} \,(\,\m+\,y \,)^p
-\,\frac{2\,\m}{p-\,1}\,(\,\m+\,y \,)^{p-\,1}\,.
\end{equation}
in particular
\begin{equation}\label{ges}%
\frac{1}{p} \,(\,\m+\,y \,)^p -\,C_1\,\mu^2 \leq\,G(y^2) \leq
\,\frac{2}{p} \,(\,\m+\,y \,)^p
\leq\,\frac{2^p}{p}\,(\,y^p+\,\mu^p\,)\,,
\end{equation}
where $ C_1\,=\,\frac{2^p}{p\,(p-\,1)}\,.$
\end{lemma}
The second inequality \eqref{ges} follows by setting $z=\,\m+\,y
\,,$where $\,z \geq\,0\,,$ and by writing
$$
\frac{2}{p}\,z^p -\,\frac{2\,\m}{p-\,1}\,z^{p-\,1}=\,
\frac{1}{p}\,z^p +\,\big(\,\frac{1}{p}\,z^p
-\,\frac{2\,\m}{p-\,1}\,z^{p-\,1}\,\big)\,.
$$
The minimum of the function between open brackets is attained for
$\,z=\,2\,\mu\,$.\par%

It readily follows that the estimates \eqref{primas},
\eqref{segasg}, \eqref{tercasg}, and \eqref{funds2} hold by setting,
for instance,

$$
 c_0=\,\frac1p\,, \hspace{0.2cm}  c_1=\,C_1\,\mu^2\,,\hspace{0.2cm} \tc_0=\,
\tc_1=\,\frac{2^p}{p}\,.
$$
In the singular case \eqref{NSC-velh-s} one has $\,G(\,y^2)=
\,\frac{2}{p}\,y^p\,.$ Straightforward manipulations lead to the
estimates claimed in the theorem \ref{isadeo2}.

\end{document}